\documentclass[reqno]{amsart}

\usepackage{amsmath}
\usepackage{graphicx}
\usepackage{amscd}
\usepackage{amsmath, amssymb}
\usepackage{graphics}
\usepackage{graphicx}
\usepackage{epsfig}
\usepackage{xcolor}
\usepackage{soul}
\usepackage{tikz}

\newtheorem{theorem}{Theorem}

\newtheorem{definition}[theorem]{Definition}

\newtheorem{lemma}[theorem]{Lemma}

\newtheorem{question}[theorem]{Question}

\begin{document}

\title[On extensions of the standard representation of $B_n$ to $SB_n$]{On extensions of the standard representation of the braid group to the Singular Braid group}

\author{Mohamad N. Nasser}

\address{Mohamad N. Nasser\\
         Department of Mathematics and Computer Science\\
         Beirut Arab University\\
         P.O. Box 11-5020, Beirut, Lebanon}
         
\email{m.nasser@bau.edu.lb}

\begin{abstract}
For an integer $n \geq 2$, set $B_n$ to be the braid group on $n$ strands and $SB_n$ to be the singular braid group on $n$ strands. $SB_n$ is one of the important group extensions of $B_n$ that appeared in 1998. Our aim in this paper is to extend the well-known standard representation of $B_n$, namely $\rho_S:B_n \to GL_n(\mathbb{Z}[t^{\pm 1}])$, to $SB_n$, for all $n \geq 2$, and to investigate the characteristics of these extended representations as well. The first major finding in our paper is that we determine the form of all representations of $SB_n$, namely $\rho'_S: SB_n \to GL_n(\mathbb{Z}[t^{\pm 1}])$, that extend $\rho_S$, for all $n\geq 2$. The second major finding is that we find necessary and sufficient conditions for the irreduciblity of the representations of the form $\rho'_S$ of $SB_n$, for all $n\geq 2$. We prove that, for $t\neq 1$, the representations of the form $\rho'_S$ are irreducible and, for $t=1$, the representations of the form $\rho'_S$ are irreducible if and only if $a+c\neq 1.$ The third major result is that we consider the virtual singular braid group on $n$ strands, $VSB_n$, which is a group extension of both $B_n$ and $SB_n$, and we determine the form of all representations $\rho''_S: VSB_2 \to GL_2(\mathbb{Z}[t^{\pm 1}])$, that extend $\rho_S$ and $\rho'_S$; making a path toward finding the form of all representations $\rho''_S: VSB_n \to GL_n(\mathbb{Z}[t^{\pm 1}])$, that extend $\rho_S$ and $\rho'_S$, for all $n\geq 3$.
\end{abstract}

\maketitle
\renewcommand{\thefootnote}{}

\vspace*{-1cm}

\section{Introduction} 

The braid group on $n$ strands, $B_n$, is a fundamental algebraic structure that has connections to many areas of Mathematics such as Representation Theory, Knot Theory and Quantum Computing \cite{K.Mu, C.DE}. It was introduced by E. Artin in 1925 \cite{E.A}, where it was defined by its generators $\sigma_1,\sigma_2, \ldots,\sigma_{n-1}$ with the following defining relations.
\begin{align*}
&\sigma_i\sigma_{i+1}\sigma_i = \sigma_{i+1}\sigma_i\sigma_{i+1} ,\hspace{0.65cm} i=1,2,\ldots,n-2,\\
&\ \ \hspace*{0.05cm} \ \ \ \sigma_i\sigma_j = \sigma_j\sigma_i , \hspace{1.65cm} |i-j|\geq 2.
\end{align*}

The pure braid group, $P_n$, is a normal subgroup of $B_n$, which is defined as the kernel of the homomorphism $B_n \rightarrow S_n$ defined by $\sigma_i \rightarrow (i \hspace{0.2cm} i+1)$, $1\leq i \leq n-1$, where $S_n$ is the symmetric group of $n$ elements (see \cite{E.A.2}). It admits a presentation with the following generators.
$$A_{ij}=\sigma_{j-1}\sigma_{j-2}\ldots \sigma_{i+1}\sigma^2_{i}\sigma^{-1}_{i+1}\ldots \sigma^{-1}_{j-2}\sigma^{-1}_{j-1}, \hspace{0.5cm} 1\leq i<j\leq n.$$

There are a lot of generalizations of $B_n$, both from algebraic and geometric point of view. One of these generalizations is the singular braid monoid, $SM_n$, which was introduced by J. Baez in 1992 \cite{J.Ba} and J. Birman in 1993 \cite{J.Bi} independently. It is a monoid generated by the generators $\sigma_1^{\pm 1},\sigma_2^{\pm 1}, \ldots,\sigma_{n-1}^{\pm 1}$ of $B_n$ and a family of singular generators $\tau_1,\tau_2, \ldots, \tau_{n-1}$. In 1998, R. Fenn, E. Keyman and C. Rourke proved that $SM_n$ embeds into a group, denoted by $SB_n$, and called the singular braid group \cite{R.F}, which has the same generators and defining relations as $SM_n$.

\vspace*{0.1cm}

On the other hand, one of the monoid extensions of both $B_n$ and $SM_n$ is the virtual singular braid monoid, $VSM_n$, which was introduced by C. Caprau, A. Pena and S. McGahan in 2016 \cite{C.Ca}. $VSM_n$ is a monoid generated by the generators $\sigma_1^{\pm 1},\sigma_2^{\pm 1}, \ldots,\sigma_{n-1}^{\pm 1}$, $\tau_1,\tau_2, \ldots, \tau_{n-1}$ of $SM_n$ in addition to a family of non-singular generators $\nu_1^{\pm 1},\nu_2^{\pm 1}, \ldots, \nu_{n-1}^{\pm 1}$. It was shown in 2022 by C. Caprau and A. Yeung that $VSM_n$ embeds into a group, denoted by $VSB_n$, and called the virtual singular braid group \cite{C.Ca2022}, which has the same generators and defining relations as $VSM_n$.

\vspace*{0.1cm}

One of the most important questions to be examined in this field is to study the irreducibility of representations of many algebraic structures. A representation $\rho$ of a group $G$ on a vector space $V$ is said to be irreducible if there is no non-trivial subspace of $V$ which is invariant under $\rho$. Another important question to be examined for representations is to study their faithfulness. The importance of finding a faithful representation of a group $G$ is to prove that $G$ is linear.

\vspace*{0.1cm}

There are many famous representations of the braid group $B_n$ where the question of irreducibility and faithfulness is answered. One of these representations is Burau representation \cite{W.B} and the other is Lawrence-Krammer-Bigelow representation \cite{R.Law90}. Regarding the irreducibility of these two representations, it has been proven by E. Formanek in \cite{E.F} that Burau representation is reducible, and by C. Levaillant in \cite{C.Le} that Lawrence-Krammer-Bigelow representation is irreducible. On the other hand, regarding the faithfulness, Burau representation was proved to be faithful for $n\leq 3$ in \cite{J.Bi1974} and unfaithful for $n\geq 5$ in \cite{J.Moo, D.long, S.Big}; whereas the case $n=4$ remains open. In contrast, it was shown in \cite{S.Big2001} and \cite{D.Kram2002} that Lawrence-Krammer-Bigelow representation is faithful for all $n\geq 2$; thus, $B_n$ is shown to be linear.

\vspace*{0.1cm}

Lately, it was interested in extending the known representations of $B_n$ to its group and monoid extensions, and to examine the irreducibility and the faithfulness of these extended representations. An important task to do here is to study the ways of extending representations of $B_n$ to its extensions. One of these ways, for example, is the $k$-local extensions way, which is defined by M. Nasser in \cite{M.N20241}. A lot of work has been done in classifying this type of extensions to many group and monoid extensions of $B_n$ (see \cite{T.M, M.M.M}). More specifically, another way for extending representations of $B_n$ to $SM_n$ and $SB_n$ was found by V. Bardakov, N. Chbili, and T. Kozlovskaya in \cite{BCK}; we call it the $\Phi$-type extensions way. One of the extensions in the $\Phi$-type way is the Birman representation of $SM_n$ defined in \cite{J.Bi}, that was shown to be faithful by L. Paris in \cite{L.Pa}, which implies that $SM_n$ is linear. M. Nasser studied the faithfulness of this type of extensions of representations of $B_n$ to $SM_n$ in some cases \cite{M.N20240}. Also, M. Prabhakar and N. Komal extended the concept of $\Phi$-type extensions to one of the monoid extensions of both $B_n$ and $SM_n$, namely the singular twisted virtual braid monoid \cite{Pra2025}.

\vspace*{0.1cm}

In this paper, we are interested in another important representation of $B_n$, which is called the standard representation (or Tong-Yang-Ma representation), that introduced first by D. Tong, S. Yang and Z. Ma in 1996 \cite{D.T}. Necessary and sufficient conditions for the irreducibility of the complex specialization of this representation has been done by I. Sysoeva in \cite{I.S}. She proved that the standard representation $\rho_S: B_n\to GL_n(\mathbb{C})$ is irreducible if and only if $t\neq 1$. Regarding the faithfulness, the standard representation $\rho_S: B_n\to GL_n(\mathbb{Z}[t^{\pm 1}])$, where $t$ is indeterminate, was proved to be unfaithful for all $n\geq 3$ (see \cite{G.M}, \cite{C.B}).

\vspace*{0.1cm}

Recently, in \cite{A.So}, A. Soulie and A. Takano used the method of D. Silver and S. Williams introduced in \cite{D.Si} to extend the standard representation $\rho_S$ to the string links and welded string links. We aim, in our article, to extend the standard representation $\rho_S$ to two group extensions of the braid group $B_n$: the singular braid group $SB_n$ and the virtual singular braid group $VSB_n$. Moreover, we examine the irreducibility and the faithfulness of these extensions. The paper is organized as follows. In Section 2, we give the main definitions and relations for the targeted groups and monoids in this paper: $B_n$, $SM_n$, $SB_n$, $VSM_n$ and $VSB_n$, and we give the definition and the main characteristics of the standard representation $\rho_S$ of $B_n$. In section 3, we find the forms of all extensions of the standard representation $\rho_S$ of $B_n$ to $SB_n$, for all $n\geq 2$ (Theorem \ref{ThSB2} and Theorem \ref{ThSB3}). In section 4, we consider all extensions $\rho'_S$ of the standard representation of $B_n$ to $SB_n$. In the first subsection, we find necessary and sufficient conditions for the irreducibility of the representations $\rho'_S$ of $SB_n$, for all $n\geq 2$ (Theorem \ref{irredd}), while in the second subsection, we discuss the faithfulness of these representations, for all $n\geq 2$. At the last, making a path toward a future work, we find the forms of all extensions of the standard representation $\rho_S$ of $B_2$ to $VSB_2$ (Theorem \ref{Theexttt}). In addition, we offered a few significant open questions for further research.
  
\vspace*{0.15cm}
  
\section{Generalities and Previous Results} 

In the beginning of this section, we give the presentation of the braid group $B_n$ introduced by E. Artin in 1925.

\begin{definition} \cite{E.A}
The braid group, $B_n$, is the group defined by the Artin generators $\sigma_1,\sigma_2,\ldots,\sigma_{n-1}$ with the defining relations
\begin{equation} \label{eqs1}
\ \ \ \ \sigma_i\sigma_{i+1}\sigma_i = \sigma_{i+1}\sigma_i\sigma_{i+1} ,\hspace{0.45cm} i=1,2,\ldots,n-2,
\end{equation}
\begin{equation} \label{eqs2}
\sigma_i\sigma_j = \sigma_j\sigma_i , \hspace{1.48cm} |i-j|\geq 2.
\end{equation}
\end{definition}

Now, we give the presentation of the singular braid monoid $SM_n$ introduced by J. Baez in 1992 and J. Birman in 1993 independently.

\begin{definition} \cite{J.Ba,J.Bi}
The singular braid monoid, $SM_n$, is the monoid defined by the generators $\sigma_1^{\pm 1},\sigma_2^{\pm 1}, \ldots,\sigma_{n-1}^{\pm 1}$ of $B_n$ in addition to the singular generators $\tau_1,\tau_2, \ldots, \tau_{n-1}$. In addition to the relations (\ref{eqs1}) and (\ref{eqs2}) above, the generators $\sigma_i^{\pm 1}$ and $\tau_i$ of $SM_n$ satisfy the following relations.
\begin{align}
\label{eqs3}
& \hspace*{-0.55cm} \tau_i \tau_j = \tau_j \tau_i,\qquad \ \ \ \ \ \ \ \ |i - j| \geq 2,\\
\label{eqs4}
\tau_i \sigma_j &= \sigma_j \tau_i,\qquad \ \ \ \ \ \ \ \ |i - j| \geq 2,\\
\label{eqs5}
\tau_i \sigma_i &= \sigma_i \tau_i,\qquad \ \ \ \ \ \ \ \ i = 1, 2, \ldots, n-1,\\
\label{eqs6}
\sigma_i \sigma_{i+1} \tau_i &= \tau_{i+1} \sigma_i \sigma_{i+1}, \qquad i = 1, 2, \ldots, n-2,\\
\label{eqs7}
\sigma_{i+1} \sigma_{i} \tau_{i+1} &= \tau_{i} \sigma_{i+1} \sigma_{i},\qquad \ \ \hspace*{0.1cm}  i = 1, 2, \ldots, n-2.
\end{align}
\end{definition}

In \cite{R.F}, R. Fenn, E. Keyman and C. Rourke proved that $SM_n$ embeds into the group $SB_n$ which is called the singular braid group and has the same generators and defining relations as $SM_n$.
\vspace*{0.05cm}

Next, we give the presentation of the virtual singular braid monoid $VSM_n$ introduced by C. Caprau, A. Pena and S. McGahan in 2016.

\begin{definition} \cite{C.Ca}
The virtual singular braid monoid, $VSM_n$, is the monoid defined by the generators $\sigma_1^{\pm 1},\sigma_2^{\pm 1}, \ldots,\sigma_{n-1}^{\pm 1}$ and $\tau_1,\tau_2, \ldots, \tau_{n-1}$ of $SM_n$ in addition to a new family of non-singular generators, namely $\nu_1^{\pm 1},\nu_2^{\pm 1},\ldots, \nu_{n-1}^{\pm 1}$. In addition to the relations (\ref{eqs1}),\ldots,(\ref{eqs7}) above, the generators $\sigma_i^{\pm 1},$ $\tau_i$, and $\nu_i^{\pm 1}$ of $VSM_n$ satisfy the following relations.
\begin{equation} \label{eqs81}
\ \ \ \  \ \ \ \ \ \ \ \nu_i^2=1 ,\hspace{1.85cm} i=1,2,\ldots, n-1,
\end{equation}
\begin{equation} \label{eqs91}
\ \ \ \ \nu_i\nu_{i+1}\nu_i = \nu_{i+1}\nu_i\nu_{i+1} ,\hspace{0.45cm} i=1,2,\ldots,n-2,
\end{equation}
\begin{equation} \label{eqs101}
\ \ \ \ \nu_i\sigma_{i+1}\nu_i = \nu_{i+1}\sigma_i\nu_{i+1} ,\hspace{0.45cm} i=1,2,\ldots,n-2,
\end{equation}
\begin{equation} \label{eqs111}
\ \ \ \ \nu_i\tau_{i+1}\nu_i = \nu_{i+1}\tau_i\nu_{i+1} ,\hspace{0.5cm} i=1,2,\ldots,n-2,
\end{equation}
\begin{equation} \label{eqs121}
\nu_i\sigma_j=\sigma_j\nu_i ,\hspace{1.4cm} |i-j|\geq 2,
\end{equation}
\begin{equation} \label{eqs131}
\nu_i\tau_j=\tau_j\nu_i ,\hspace{1.5cm} |i-j|\geq 2.
\end{equation}
\end{definition}

In \cite{C.Ca2022}, C. Caprau and A. Yeung showed that $VSM_n$ embeds into the group $VSB_n$ which is called the virtual singular braid group and has the same generators and defining relations as $VSM_n$.
\vspace*{0.1cm}

Now, we give the concept of an important type of representations of $B_n$, namely local representations.

\begin{definition}
A representation $\rho: B_n \rightarrow GL_{m}(\mathbb{Z}[t^{\pm 1}])$ is said to be local if it has the form
$$\rho(\sigma_i) =\left( \begin{array}{c|@{}c|c@{}}
   \begin{matrix}
     I_{i-1} 
   \end{matrix} 
      & 0 & 0 \\
      \hline
    0 &\hspace{0.2cm} \begin{matrix}
   		M_i
   		\end{matrix}  & 0  \\
\hline
0 & 0 & I_{n-i-1}
\end{array} \right) \hspace*{0.2cm} \text{for} \hspace*{0.2cm} 1\leq i\leq n-1,$$ 
where $M_i \in GL_k(\mathbb{Z}[t^{\pm 1}])$ with $k=m-n+2$ and $I_r$ is the $r\times r$ identity matrix. The representation $\rho$ is said to be homogeneous if all the matrices $M_i$ are equal.
\end{definition}

In the next definition, we give the concept of local representations of $SM_n$ and $SB_n$. In a typical way, the concept of local representations of $VSM_n$ and $VSB_n$ could be defined also.

\begin{definition} \label{LocalSMn}
A representation $\rho: SM_n \rightarrow GL_{m}(\mathbb{Z}[t^{\pm 1}])$ is said to be local if it has the form
$$\rho(\sigma_i)=\left( \begin{array}{c|@{}c|c@{}}
   \begin{matrix}
     I_{i-1} 
   \end{matrix} 
      & 0 & 0 \\
      \hline
    0 &\hspace{0.2cm} \begin{matrix}
   		M_i
   		\end{matrix}  & 0  \\
\hline
0 & 0 & I_{n-i-1}
\end{array} \right) \hspace*{0.2cm} \text{for} \hspace*{0.2cm} 1\leq i\leq n-1,$$
and
$$\rho(\tau_i) =\left( \begin{array}{c|@{}c|c@{}}
   \begin{matrix}
     I_{i-1} 
   \end{matrix} 
      & 0 & 0 \\
      \hline
    0 &\hspace{0.2cm} \begin{matrix}
   		N_i
   		\end{matrix}  & 0  \\
\hline
0 & 0 & I_{n-i-1}
\end{array} \right) \hspace*{0.2cm} \text{for} \hspace*{0.2cm} 1\leq i\leq n-1,$$ 
where $M_i \in GL_k(\mathbb{Z}[t^{\pm 1}])$ and $N_i \in M_k(\mathbb{Z}[t^{\pm 1}])$ with $k=m-n+2$ and $I_r$ is the $r\times r$ identity matrix. The representation $\rho$ is said to be homogeneous if all the matrices $M_i$ are equal and all the matrices $N_i$ are equal. If in addition $\rho(\tau_i)$ is invertible in $GL_{m}(\mathbb{Z}[t^{\pm 1}])$ for all $1\leq i \leq n-1$, then $\rho$ becomes a representation of $SB_n$.
\end{definition}

The representations given in the following definitions are examples of the homogeneous local representations of the braid group $B_n$ to $GL_m(\mathbb{Z}[t^{\pm 1}])$ in two cases: $m=n$ and $m=n+1$.
 
\begin{definition} \cite{W.B} \label{defBurau}
The Burau representation $\rho_B: B_n\rightarrow GL_n(\mathbb{Z}[t^{\pm 1}])$ is the representation defined by
$$\sigma_i\rightarrow \left( \begin{array}{c|@{}c|c@{}}
   \begin{matrix}
     I_{i-1} 
   \end{matrix} 
      & 0 & 0 \\
      \hline
    0 &\hspace{0.2cm} \begin{matrix}
   	1-t & t\\
   	1 & 0\\
\end{matrix}  & 0  \\
\hline
0 & 0 & I_{n-i-1}
\end{array} \right) \hspace*{0.2cm} \text{for} \hspace*{0.2cm} 1\leq i\leq n-1.$$ 
\end{definition}

\begin{definition} \cite{V.B2016} \label{defF}
The $F$-representation $\rho_F: B_n \rightarrow GL_{n+1}(\mathbb{Z}[t^{\pm 1}])$ is the representation defined by
$$\sigma_i\rightarrow \left( \begin{array}{c|@{}c|c@{}}
   \begin{matrix}
     I_{i-1} 
   \end{matrix} 
      & 0 & 0 \\
      \hline
    0 &\hspace{0.2cm} \begin{matrix}
   		1 & 1 & 0 \\
   		0 &  -t & 0 \\   		
   		0 &  t & 1 \\
   		\end{matrix}  & 0  \\
\hline
0 & 0 & I_{n-i-1}
\end{array} \right) \hspace*{0.2cm} \text{for} \hspace*{0.2cm} 1\leq i\leq n-1.$$ 
\end{definition}

For more information on the irreducibility of the Burau representation and the $F$-representation, see \cite{E.F} and \cite{M.N2024} respectively.
\vspace*{0.1cm}

Another famous local representation of $B_n$ that is given in the following definition is called the standard representation. The standard representation was introduced first by D. Tong, S. Yang and Z. Ma in 1996.  

\begin{definition}\cite{D.T}
Let $t$ be indeterminate. The standard representation $\rho_S: B_n\to GL_n(\mathbb{Z}[t^{\pm 1}])$ is the representation given by
$$\sigma_i \rightarrow \left( \begin{array}{c|@{}c|c@{}}
   \begin{matrix}
     I_{i-1} 
   \end{matrix} 
      & 0 & 0 \\
      \hline
    0 &\hspace{0.2cm} \begin{matrix}
   	0 & t\\
   	1 & 0\\
\end{matrix}  & 0  \\
\hline
0 & 0 & I_{n-i-1}
\end{array} \right) \hspace*{0.2cm} \text{for} \hspace*{0.2cm} 1\leq i \leq n-1.$$ 
\end{definition}

In the following theorem, we specialize $t$ to be a non-zero complex number and we set a previous result done by I. Sysoeva regarding irreducibility of the standard representation $\rho_s$.

\begin{theorem} \cite{I.S} \label{irred}
The standard representation $\rho_S: B_n\to GL_n(\mathbb{C})$ is irreducible if and only if $t\neq 1$
\end{theorem}

On the other hand, the following theorem is a previous result done by C. Blanchet and I. Marin regarding the faithfulness of the standard representation $\rho_S$.

\begin{theorem} \cite{C.B} \label{faith}
The standard representation $\rho_S: B_n\to GL_n(\mathbb{Z}[t^{\pm 1}])$ is unfaithful for all $n\geq 3$.
\end{theorem}

\vspace*{0.1cm}

In order to define new representations of any algebraic structure that extends $B_n$, it is naturally to start by extending the known representations of $B_n$ to these algebraic structures. We may have many different ways of extensions. In \cite{M.N20241}, M. Nasser considered two types of extensions of braid group representations to $SM_n$ in order to compare between them. The first type is the local extension and the second type is the $\Phi$-type extension. He made his study on the three known representations of $B_n$ defined above: Burau representation, the $F$-representation and the standard representation. This was a road toward answering the following question.

\begin{question} \label{qqq}
Let $\rho: B_n \rightarrow GL_{m}(\mathbb{Z}[t^{\pm 1}])$ be a representation. Is there a relation between the types of extensions of $\rho$ to any algebraic structure that extends $B_n$?
\end{question}

In the next section, we answer this question for the standard representation $\rho_S$. 

\section{Extensions of the Standard Representation $\rho_s$ of $B_n$ to $SB_n$} 

In this section, we aim to determine the form of all possible representations $\rho'_S: SB_n\to GL_n(\mathbb{Z}[t^{\pm 1}])$ that extend the standard representation $\rho_S: B_n\to GL_n(\mathbb{Z}[t^{\pm 1}])$, for all $n\geq 2$. We start in the following theorem with $n=2$, which is a special case.

\begin{theorem} \label{ThSB2}
Let $t$ be indeterminate and let $\rho'_S: SB_2\to GL_2(\mathbb{Z}[t^{\pm 1}])$ be a representation that extends the standard representation $\rho_S: B_2\to GL_2(\mathbb{Z}[t^{\pm 1}])$. Then, $\rho'_S$ acts on the generators $\sigma_1$ and $\tau_1$ of $SB_2$ as follows.
$$\rho'_S(\sigma_1)=\rho_S(\sigma_1)=\left(
  \begin{matrix}
   	0 & t\\
   	1 & 0\\
\end{matrix} \right) \text{ and } \rho'_S(\tau_1)=\left(
  \begin{matrix}
   	a & ct\\
   	c & a\\
\end{matrix} \right), \text{ where } a,c \in \mathbb{Z}[t^{\pm 1}].$$ 
\end{theorem}
\begin{proof}
As $\rho'_S$ extends $\rho_S$, we get $\rho'_S(\sigma_1)=\rho_S(\sigma_1)=\left(
  \begin{matrix}
   	0 & t\\
   	1 & 0\\
\end{matrix} \right)$. Regarding the generator $\tau_1$ of $SB_2$, set $\rho'_S(\tau_1)=\left(
  \begin{matrix}
   	a & b\\
   	c & d\\
\end{matrix} \right), \text{ where } a,b,c,d \in \mathbb{Z}[t^{\pm 1}].$ The only relation between the generators of $SB_2$ is: $\sigma_1\tau_1=\tau_1\sigma_1,$ which implies that $$\rho'_S(\sigma_1)\rho'_S(\tau_1)=\rho'_S(\tau_1)\rho'_S(\sigma_1).$$ Applying this equality we get the following system of equations.
\begin{equation}
b-ct=0
\end{equation}
\begin{equation}
at-dt=0
\end{equation} 
\begin{equation}
a-d=0
\end{equation}
Solving this system of equations implies that $a=d$ and $b=ct$, as required.
\end{proof}

Now, we generalize the result of the previous theorem to $SB_n$ for all $n\geq 3$.

\begin{theorem} \label{ThSB3}
Consider $n\geq 3$ and let $t$ be indeterminate. Let $\rho'_S: SB_n\to GL_n(\mathbb{Z}[t^{\pm 1}])$ be a representation that extends the standard representation $\rho_S: B_n\to GL_n(\mathbb{Z}[t^{\pm 1}])$. Then, $\rho'_S$ acts on the generators $\sigma_i$ and $\tau_i$, $1\leq i \leq n-1,$ of $SB_n$ as follows.
$$\rho'_S(\sigma_i)=\rho_S(\sigma_i)=\left( \begin{array}{c|@{}c|c@{}}
   \begin{matrix}
     I_{i-1} 
   \end{matrix} 
      & 0 & 0 \\
      \hline
    0 &\hspace{0.2cm} \begin{matrix}
   	0 & t\\
   	1 & 0\\
\end{matrix}  & 0  \\
\hline
0 & 0 & I_{n-i-1}
\end{array} \right)$$
and
$$\rho'_S(\tau_i)=\left( \begin{array}{c|@{}c|c@{}}
   \begin{matrix}
     I_{i-1} 
   \end{matrix} 
      & 0 & 0 \\
      \hline
    0 &\hspace{0.2cm} \begin{matrix}
   	a & ct\\
   	c & a\\
\end{matrix}  & 0  \\
\hline
0 & 0 & I_{n-i-1}
\end{array} \right),$$
where $a,c \in \mathbb{Z}[t^{\pm 1}]$.
\end{theorem}
\begin{proof}
As the proof is computational and similar, we give the proof for the case $n=3$, and the proof could be generalized for all $n>3$ in a similar procedure.

Now, since $\rho'_S$ is an extension of $\rho_S$, we get that $\rho'_S(\sigma_i)=\rho_S(\sigma_i)$ for $i=1,2$. Regarding the generators $\tau_1$ and $\tau_2$ of $SB_3$, set $$\rho'_S(\tau_1)=\left(
  \begin{matrix}
   	a_1 & b_1 & c_1\\
   	d_1 & e_1 & f_1\\
   	g_1 & h_1 & i_1\\
\end{matrix} \right) \text{ and } \rho'_S(\tau_2)=\left(
  \begin{matrix}
   	a_2 & b_2 & c_2\\
   	d_2 & e_2 & f_2\\
   	g_2 & h_2 & i_2\\
\end{matrix} \right),$$
where $a_1,b_1,c_1,d_1,e_1,f_1,g_1,h_1,i_1,a_2,b_2,c_2,d_2,e_2,f_2,g_2,h_2,i_2 \in \mathbb{Z}[t^{\pm 1}]$.
The relations of the generators of $SB_3$ that contains the generators $\tau_i$, $1\leq i \leq 2$, are:
$$\sigma_1\tau_1=\tau_1\sigma_1,$$ 
$$\sigma_2\tau_2=\tau_2\sigma_2,$$ 
$$\sigma_1\sigma_2\tau_1=\tau_2\sigma_1\sigma_2,$$
$$\sigma_2\sigma_1\tau_2=\tau_1\sigma_2\sigma_1.$$
Applying these relations of the images of generators of $SB_3$ under $\rho'_S$ implies the following system of thirty two equations and eighteen unknowns.
\begin{equation}\label{eqs11}
-b_1+d_1t=0
\end{equation}
\begin{equation}\label{eqs12}
-a_1t+e_1t=0
\end{equation}
\begin{equation}\label{eqs13}
-c_1+f_1t=0
\end{equation}
\begin{equation}\label{eqs14}
a_1-e_1=0
\end{equation}
\begin{equation}\label{eqs15}
c_1-f_1=0
\end{equation}
\begin{equation}\label{eqs16}
g_1-h_1=0
\end{equation}
\begin{equation}\label{eqs17}
h_1-g_1t=0
\end{equation}
\begin{equation}\label{eqs18}
b_2-c_2=0
\end{equation}
\begin{equation}\label{eqs19}
c_2-b_2t=0
\end{equation}
\begin{equation}\label{eqs20}
-d_2+g_2t=0
\end{equation}
\begin{equation}\label{eqs21}
-f_2+h_2t=0
\end{equation}
\begin{equation}\label{eqs22}
-e_2t+i_2t=0
\end{equation}
\begin{equation}\label{eqs23}
d_2-g_2=0
\end{equation}
\begin{equation}\label{eqs24}
e_2-i_2=0
\end{equation}
\begin{equation}\label{eqs25}
-b_2+g_1t^2=0
\end{equation}
\begin{equation}\label{eqs26}
-c_2+h_1t^2=0
\end{equation}
\begin{equation}\label{eqs27}
-a_2t^2+i_1t^2=0
\end{equation}
\begin{equation}\label{eqs28}
a_1-e_2=0
\end{equation}
\begin{equation}\label{eqs29}
b_1-f_2=0
\end{equation}
\begin{equation}\label{eqs30}
c_1-d_2t^2=0
\end{equation}
\begin{equation}\label{eqs31}
d_1-h_2=0
\end{equation}   
\begin{equation}\label{eqs32}
e_1-i_2=0
\end{equation}  
\begin{equation}\label{eqs33}
f_1-g_2t^2=0
\end{equation}
\begin{equation}\label{eqs34}
-c_1+d_2t=0
\end{equation}
\begin{equation}\label{eqs35}
-a_1t+e_2t=0
\end{equation}
\begin{equation}\label{eqs36}
-b_1t+f_2t=0
\end{equation}
\begin{equation}\label{eqs37}
-f_1+g_2t=0
\end{equation}
\begin{equation}\label{eqs38}
-d_1t+h_2t=0
\end{equation}
\begin{equation}\label{eqs39}
-e_1t+i_2t=0
\end{equation}
\begin{equation}\label{eqs40}
a_2-i_1=0
\end{equation}
\begin{equation}\label{eqs41}
b_2-g_1t=0
\end{equation}
\begin{equation}\label{eqs42}
c_2-h_1t=0
\end{equation}
We solve this system as follows.
\begin{itemize}
\item Equation (\ref{eqs11}) implies that $b_1=d_1t$.\vspace*{0.05cm}
\item Equations (\ref{eqs12}), (\ref{eqs14}), (\ref{eqs28}), (\ref{eqs32}) and (\ref{eqs39}) implies that $a_1=e_1=e_2=i_2$.\vspace*{0.05cm}
\item Equations (\ref{eqs13}) and (\ref{eqs15}) implies that $c_1=f_1=0$.\vspace*{0.05cm}
\item Equations (\ref{eqs16}) and (\ref{eqs17}) implies that $g_1=h_1=0$.\vspace*{0.05cm}
\item Equations (\ref{eqs18}) and (\ref{eqs19}) implies that $b_2=c_2=0$.\vspace*{0.05cm}
\item Equations (\ref{eqs20}) and (\ref{eqs23}) implies that $d_2=g_2=0$.\vspace*{0.05cm}
\item Equation (\ref{eqs21}) implies that $f_2=h_2t$.\vspace*{0.05cm}
\item Equations (\ref{eqs22}) and (\ref{eqs24}) implies that $e_2=i_2$.\vspace*{0.05cm}
\item Equations (\ref{eqs27}) and (\ref{eqs40}) implies that $a_2=i_1$.\vspace*{0.05cm}
\item Equations (\ref{eqs29}) and (\ref{eqs36}) implies that $b_1=f_2$.\vspace*{0.05cm}
\item Equations (\ref{eqs31}) and (\ref{eqs38})implies that $d_1=h_2$.
\end{itemize}
Now, substituting these obtained values in $\rho'_S(\tau_1)$ and $\rho'_S(\tau_2)$ completes the proof.
\end{proof}

\section{On the Irreducibility and the Faithfulness of the Representations of the Form $\rho'_s$ of $SB_n$} 

\subsection{On the Irreducibility of $\rho'_s$}
In the first part of this section, we specify $t$ to be a non-zero complex number. Recall that, according to Theorem \ref{irred}, the standard representation $\rho_S: B_n\to GL_n(\mathbb{C})$ is irreducible if and only if $t\neq 1$. In this subsection, we find necessary and sufficient conditions for the irreducibility of the representations of the form $\rho'_S: SB_n \to GL_n(\mathbb{C})$ determined in Section 3. First, we give two lemmas that helps us in our proof.

\begin{lemma} \label{lemma0}
Consider $n\geq 2$ and let $\rho: G \to GL_n(V)$ be a representation of a group $G$ on a vector space $V$. Let $H$ be any subgroup of $G$. If $\rho: H \to GL_n(V)$ is irreducible, then, $\rho: G \to GL_n(V)$ is also irreducible.
\end{lemma}
\begin{proof}
Suppose that $\rho: G \to GL_n(V)$ is reducible, then, there exists an invariant subspace $S$ of $V$ such that $\rho(g)s\in S$ for every $s\in S$ and $g\in G$. This implies that $\rho(h)s\in S$ for every $s\in S$ and $h\in H\subset G$. Hence, $S$ is an invariant subspace of $V$ with respect to $\rho: H \to GL_n(V)$, which is a contradiction.
\end{proof}

\begin{lemma} \label{lemma}
Let $\rho'_S: SB_n \to GL_n(\mathbb{C})$ be an extension of the standard representation of $B_n$ determined in Section 3 evaluated at $t=1$. Let $S$ be a subspace of $\mathbb{C}^n$ that is invariant under $\rho'_S$. If there exists $1\leq i \leq n$ such that $e_i \in S$, where $\{e_1,e_2,\ldots, e_n\}$ is the canonical basis of $\mathbb{C}^n$, then $S=\mathbb{C}^n$.
\end{lemma}
\begin{proof}
Recall that the representation $\rho'_S: SB_n \to GL_n(\mathbb{C})$ evaluated at $t=1$ becomes as follows.
$$\rho'_S(\sigma_i)=\left( \begin{array}{c|@{}c|c@{}}
   \begin{matrix}
     I_{i-1} 
   \end{matrix} 
      & 0 & 0 \\
      \hline
    0 &\hspace{0.2cm} \begin{matrix}
   	0 & 1\\
   	1 & 0\\
\end{matrix}  & 0  \\
\hline
0 & 0 & I_{n-i-1}
\end{array} \right)$$
and
$$\rho'_S(\tau_i)=\left( \begin{array}{c|@{}c|c@{}}
   \begin{matrix}
     I_{i-1} 
   \end{matrix} 
      & 0 & 0 \\
      \hline
    0 &\hspace{0.2cm} \begin{matrix}
   	a & c\\
   	c & a\\
\end{matrix}  & 0  \\
\hline
0 & 0 & I_{n-i-1}
\end{array} \right),$$
where $a,c \in \mathbb{C}$ with $a^2-c^2 \neq 0$. We consider the following cases.
\begin{itemize}
\item[(1)] \underline{Case $i=1$}: If $e_1\in S$, then $\rho'_S(\sigma_1)(e_1)=e_2\in S$, and so, $\rho'_S(\sigma_2)(e_2)=e_3\in S$. Continue in this way, we get that $e_i\in S$ for all $1\leq i \leq n$, and so, $S=\mathbb{C}^n$, as required.
\item[(2)] \underline{Case $i=n$}: If $e_n\in S$, then $\rho'_S(\sigma_{n-1})(e_n)=e_{n-1}\in S$, and so, $\rho'_S(\sigma_{n-2})(e_{n-1})=e_{n-2}\in S$. Continue in this way, we get again that $e_i\in S$ for all $1\leq i \leq n$, and so, $S=\mathbb{C}^n$, as required.
\item[(3)] \underline{Case $1<i<n$}: If $e_i\in S$ in this case, then $\rho'_S(\sigma_i)(e_i)=e_{i+1}\in S$, and so, $\rho'_S(\sigma_{i+1})(e_{i+1})=e_{i+2}\in S$. Continue in this way we get that $e_j\in S$ for all $i\leq j \leq n$. On the other hand, we have $\rho'_S(\sigma_{i-1})(e_i)=e_{i-1}\in S$, and so, $\rho'_S(\sigma_{i-2})(e_{i-1})=e_{i-2}\in S$. Continue in this way, we get that $e_i\in S$ for all $1\leq j \leq i$. Therefore, $e_i\in S$ for all $1\leq i \leq n$, and so, $S=\mathbb{C}^n$, as required.
\end{itemize}
\end{proof}

\begin{theorem} \label{irredd}
Let $\rho'_S: SB_n \to GL_n(\mathbb{C})$ be an extension of the standard representation of $B_n$ determined in Section 3. Recall that
$$\rho'_S(\sigma_i)=\left( \begin{array}{c|@{}c|c@{}}
   \begin{matrix}
     I_{i-1} 
   \end{matrix} 
      & 0 & 0 \\
      \hline
    0 &\hspace{0.2cm} \begin{matrix}
   	0 & t\\
   	1 & 0\\
\end{matrix}  & 0  \\
\hline
0 & 0 & I_{n-i-1}
\end{array} \right)$$
and
$$\rho'_S(\tau_i)=\left( \begin{array}{c|@{}c|c@{}}
   \begin{matrix}
     I_{i-1} 
   \end{matrix} 
      & 0 & 0 \\
      \hline
    0 &\hspace{0.2cm} \begin{matrix}
   	a & ct\\
   	c & a\\
\end{matrix}  & 0  \\
\hline
0 & 0 & I_{n-i-1}
\end{array} \right),$$
where $a,c \in \mathbb{C}$ with $a^2-tc^2 \neq 0$. Then, the following holds true.
\begin{itemize}
\item[(1)] If $t\neq 1$, then $\rho'_S$ is irreducible.
\item[(2)] If $t=1$, then $\rho'_S$ is irreducible if and only if $a+c\neq 1.$
\end{itemize}
\end{theorem}
\begin{proof}
We consider each case separately.
\begin{itemize}
\item[(1)] First, if $t\neq 1$, then, the restriction of $\rho'_S$ to $B_n$, which is the standard representation $\rho_S$, is irreducible by Theorem \ref{irred}. So, $\rho'_S$ is irreducible by Lemma \ref{lemma0}.
\item[(2)] Now, set $t=1$. 
\begin{itemize}
\item[$\bullet$] For the necessary condition, suppose that $a+c=1$, then, we see that the column vector $(1,1,\ldots,1)^T$ is invariant under $\rho'_S$. So, the representation $\rho'_S$ is reducible, as required.
\item[$\bullet$] For the sufficient condition, suppose that $a+c\neq 1$ and suppose to get a contradiction that $\rho'_S$ is reducible. Let $S$ be a non-trivial subspace of $\mathbb{C}^n$ that is invariant under $\rho'_S(\sigma_i)$ and $\rho'_S(\tau_i)$ for all $1\leq i \leq n-1$ and pick $x=(x_1,x_2,\ldots,x_n)^T \in S$ to be a non-zero vector. For all $1\leq i \leq n-1$, we have,
\vspace*{0.1cm}
\begin{center}
$\rho'_S(\sigma_i)(x)-x=(0,0,\ldots,0,\underbrace{x_{i+1}-x_{i}}_i,\underbrace{x_{i}-x_{i+1}}_{i+1},0,0,\ldots,0)\in S.$
\end{center}
So, if $x_i=x_{i+1}$ for all $1\leq i \leq n-1$, then, $x$ becomes the vector $(1,1,\ldots,1)^T$ multiplied by a constant, which is invariant under $\rho'_S(\tau_1)$ by our assumption, and so, we get that $a+c=1$, a contradiction.\\
\noindent Then, there exists $1\leq j \leq n-1$ such that $x_j\neq x_{j+1}$. Hence, as $\rho'_S(\sigma_j)(x)-x=(x_{j+1}-x_j)e_j+(x_j-x_{j+1})e_{j+1} \in S$, we get that the column vector
\begin{center}
$v_1=e_j-e_{j+1}\in S.$
\end{center}
\vspace*{0.1cm}
And so, we get that the column vector
\vspace*{0.1cm}
\begin{center}
$v_2=\rho'_S(\sigma_{j+1})(v_1)=e_j-e_{j+2} \in S,$
\end{center}
\vspace*{0.1cm}
which gives that the column vector
\vspace*{0.1cm}
\begin{center}
$v_3=\rho'_S(\sigma_j)(v_2)=e_{j+1}-e_{j+2} \in S$.
\end{center}
\vspace*{0.1cm}
Now, we can see that
\vspace*{0.1cm}
\begin{center}
$v_4=\rho'_S(\tau_j)(v_3)=ce_j+ae_{j+1}-e_{j+2}$,
\end{center}
\vspace*{0.1cm}
and so,
\vspace*{0.1cm}
\begin{center}
$v_5=v_4-v_3-cv_1=(a+c-1)e_{j+1}\in S$.
\end{center}
\vspace*{0.1cm}
Hence, as $a+c\neq 1$, we get that $e_{j+1}\in S$. Therefore, by Lemma \ref{lemma}, we get $S=\mathbb{C}^n$, a contradiction. Thus, $\rho'_S$ is irreducible in this case and the proof is completed.
\end{itemize}
\end{itemize}
\end{proof}
\vspace{-0.2cm}
\subsection{On the Faithfulness of $\rho'_s$}
In the second part of this section, we discuss the faithfulness of the representations of the form $\rho'_S$ determined in Section 3. For $n=2$, M. Nasser found necessary and sufficient conditions for the faithfulness of the representations of the form $\rho'_S: SB_2 \to GL_2(\mathbb{C})$ (See \cite{M.N20241}, Theorem 12). Now, for $n\geq 3$, according to Theorem \ref{faith}, the standard representation $\rho_S: B_n\to GL_n(\mathbb{Z}[t^{\pm 1}])$ is unfaithful, which directly give the following result.

\begin{theorem}
The representation $\rho'_S: SB_n\to GL_n(\mathbb{Z}[t^{\pm 1}])$ is unfaithful for all $n\geq 3$.
\end{theorem}

On the other hand, it was shown in \cite{C.B} that the commutator subgroup of the pure braid group $P_n$ is in $\ker(\rho_S)$, which implies that it is in $\ker(\rho'_S)$ for any extension $\rho'_S$ of $\rho_S$. This leads us to end this section with the following question.

\begin{question}
Let $\rho'_S: SB_n \to GL_n(\mathbb{Z}[t^{\pm 1}])$ be an extension of the standard representation $\rho_S$ of $B_n$ determined in Section 3.  What is $\ker(\rho'_S)$?
\end{question}

\section{Future Work: Extensions of the Standard Representation $\rho_s$ of $B_n$ to $VSB_n$} 

In this section, we aim to determine the shape of all possible representations $\rho''_S: VSB_n\to GL_n(\mathbb{Z}[t^{\pm 1}])$ that extend the standard representation $\rho_S: B_n\to GL_n(\mathbb{Z}[t^{\pm 1}])$ in the case $n=2$. We leave the case $n>3$ as a future work.

\begin{theorem} \label{Theexttt}
Let $t$ be indeterminate. Let $\rho''_S: VSB_2\to GL_2(\mathbb{Z}[t^{\pm 1}])$ be an extension of the standard representation of $B_2$ to $VSB_2$. Then, $\rho''_S$ is equivalent to one of the following five representations.
\begin{itemize}
\item[(1)] $\rho''_{S_1}: VSB_2 \rightarrow GL_2(\mathbb{Z}[t^{\pm 1}])$ such that
$$\rho''_{S_1}(\sigma_1)=\left(\begin{matrix}
   	0 & t\\
   	1 & 0\\
\end{matrix}\right) ,\ \rho''_{S_1}(\tau_1)=\left(
  \begin{matrix}
   	a & ct\\
   	c & a\\
\end{matrix} \right)\text{ and } \rho''_{S_1}(\nu_1)=\left( \begin{matrix}
   		p & q\\
   		\frac{1-p^2}{q} & -p
\end{matrix} \right),$$
where $a,c,p,q \in \mathbb{Z}[t^{\pm 1}]$.\\
\item[(2)] $\rho''_{S_2}: VSB_2 \rightarrow GL_2(\mathbb{Z}[t^{\pm 1}])$ such that
$$\rho''_{S_2}(\sigma_1)=\left(\begin{matrix}
   	0 & t\\
   	1 & 0\\
\end{matrix}\right),\ \rho''_{S_2}(\tau_1)=\left(
  \begin{matrix}
   	a & ct\\
   	c & a\\
\end{matrix} \right)\text{ and } \rho''_{S_2}(\nu_1)=\left( \begin{matrix}
   		-1 & 0\\
   		r & 1
\end{matrix} \right),$$
where $a,c,r \in \mathbb{Z}[t^{\pm 1}]$.\\
\item[(3)] $\rho''_{S_3}: VSB_2 \rightarrow GL_2(\mathbb{Z}[t^{\pm 1}])$ such that
$$\rho''_{S_3}(\sigma_1)=\left(\begin{matrix}
   	0 & t\\
   	1 & 0\\
\end{matrix}\right),\ \rho''_{S_3}(\tau_1)=\left(
  \begin{matrix}
   	a & ct\\
   	c & a\\
\end{matrix} \right)\text{ and } \rho''_{S_3}(\nu_1)=\left( \begin{matrix}
   		1 & 0\\
   		r & -1
\end{matrix} \right),$$
where $a,c,r \in \mathbb{Z}[t^{\pm 1}]$.\\
\item[(4)] $\rho''_{S_4}: VSB_2 \rightarrow GL_2(\mathbb{Z}[t^{\pm 1}])$ such that
$$\rho''_{S_4}(\sigma_1)=\left(\begin{matrix}
   	0 & t\\
   	1 & 0\\
\end{matrix}\right),\ \rho''_{S_4}(\tau_1)=\left(
  \begin{matrix}
   	a & ct\\
   	c & a\\
\end{matrix} \right)\text{ and } \rho''_{S_4}(\nu_1)=\left( \begin{matrix}
   		-1 & 0\\
   		0 & -1
\end{matrix} \right),$$
where $a,c \in \mathbb{Z}[t^{\pm 1}]$.\\
\item[(5)] $\rho''_{S_5}: VSB_2 \rightarrow GL_2(\mathbb{Z}[t^{\pm 1}])$ such that
$$\rho''_{S_5}(\sigma_1)=\left(\begin{matrix}
   	0 & t\\
   	1 & 0\\
\end{matrix}\right),\ \rho''_{S_5}(\tau_1)=\left(
  \begin{matrix}
   	a & ct\\
   	c & a\\
\end{matrix} \right)\text{ and } \rho''_{S_5}(\nu_1)=\left( \begin{matrix}
   		1 & 0\\
   		0 & 1
\end{matrix} \right),$$
where $a,c \in \mathbb{Z}[t^{\pm 1}]$.
\end{itemize}
\end{theorem}

\begin{proof}
Since the representation $\rho''_{S}$ is an extension of the standard representation $\rho_{S}$, and since $VSB_2$ is an extension of $B_2$ and $SB_2$, it follows directly that 
$$\rho''_{S}(\sigma_1)=\rho'_{S}(\sigma_1)=\left(\begin{matrix}
   	0 & t\\
   	1 & 0\\
\end{matrix}\right) \text{ and } \rho''_{S}(\tau_1)=\rho'_{S}(\tau_1)=\left(
  \begin{matrix}
   	a & ct\\
   	c & a\\
\end{matrix} \right), \text{ where } a,c \in \mathbb{Z}[t^{\pm 1}].$$
For the generator $\nu_1$ of $VSB_2$, set $\rho''_S(\nu_1)=\left(
  \begin{matrix}
   	p & q\\
   	r & s\\
\end{matrix} \right), \text{ where } p,q,r,s \in \mathbb{Z}[t^{\pm 1}].$ The only relation between the generators of $VSB_2$ that contains $\nu_1$ is: $\nu_1^2=1$, which implies that $\rho''_S(\nu_1^2)=I_2.$ Applying this relation gives the following system of four equations and four unknowns.
\begin{equation}
-1+p^2+qr=0
\end{equation}
\begin{equation}
pq+qs=0
\end{equation}
\begin{equation}
pr+rs=0
\end{equation}
\begin{equation}
-1+qr+s^2=0
\end{equation}
 Solving this system as we did in Theorems \ref{ThSB2} and \ref{ThSB3} implies the required results.
\end{proof}

As a future research to generalize Theorem \ref{Theexttt}, we ask the following question.

\begin{question}
Consider $n\geq 3$ and let $t$ be indeterminate. Let $\rho''_S: VSB_n\to GL_n(\mathbb{Z}[t^{\pm 1}])$ be an extension of the standard representation of $B_n$ to $VSB_n$. Then, what are the possible shapes of $\rho''_S$?
\end{question}

Also, as we examined the irreducibility of the representations that extend the standard representation of $B_n$ to $SB_n$ for all $n\geq 3$ in Section 4, one more question that could be addressed in future research is mentioned below.

\begin{question}
Consider $n\geq 2$ and let $t$ be indeterminate. Let $\rho''_S: VSB_n\to GL_n(\mathbb{Z}[t^{\pm 1}])$ be an extension of the standard representation of $B_n$ to $VSB_n$. For what conditions we get that $\rho''_S$ is irreducible?
\end{question}

\vspace{-0.3cm}

\end{document}